\documentclass{amsart}
\usepackage{amsrefs}
\usepackage{hyperref}
\usepackage{amssymb}
\usepackage{graphicx,color}
\usepackage{mathtools}

\theoremstyle{plain}
\newtheorem{theorem}{Theorem}

\newtheorem{fact}{Fact}

\theoremstyle{definition}
\newtheorem{definition}{Definition}

\newtheorem{example}{Example}

\theoremstyle{remark}
\newtheorem{remark}{Remark}

\newcommand{\StTwo}{St_{(2)}}
\newcommand{\StOne}{St_{(1)}}

\newcommand{\ind}{\operatorname{ind}}
\newcommand{\sgn}{\operatorname{sgn}}

\title[Stokes structures for Arnold-type invariants]
{Underlying Stokes and de Rham structures for Arnold-type invariants}
\author{Noboru Ito}
\address{Department of Mathematics, Faculty of Engineering, Shinshu University, Wakasato 4-17-1 Nagano, Nagano 380-8553, Japan}
\email{nito@shinshu-u.ac.jp}
\author{Hiroki Mizuno}
\address{Department of Science and Technology, Graduate School of Medicine, Science and Technology, Shinshu University, Asahi 3-1-1, Matsumoto, Nagano 390-8626, Japan}
\email{22hs602j@shinshu-u.ac.jp}
\date{May 13, 2026}

\begin{document}
\begin{abstract}
We introduce a framework on dual complexes for studying Arnold-type
invariants of immersed curves and immersed surfaces via local finite-difference
structures associated with Alexander numberings.
For generic immersed plane curves and generic immersed surfaces, we define
locally normalized maps \(d^k\phi\) on dual skeleta and show that suitable
evaluations recover the Arnold-type invariants \(St_{(1)}\) and \(St_{(2)}\).
In particular, we establish normalized discrete Stokes-type compatibilities
between adjacent dual skeleta and derive corresponding
Shumakovitch-type identities for curves and surfaces.
The normalization coefficients are determined by finite-difference factorial
structures together with multiplicities of local configurations.
We further interpret the iterated-integral-type structures appearing in
Shumakovitch-type identities through finite-difference structures and
highest-degree local Stokes compatibilities on dual complexes.
We also reinterpret the slice formula for \(St_{(2)}\) and \(St_{(1)}\)
as a compatibility relation between slicing and local operations on the
dual complex.
These results provide a unified framework in which global Arnold-type
invariants arise as distribution-type evaluations of local data on dual
complexes.
The framework further clarifies the distinction between untwisted local
closures and globally twisted structures such as the original Arnold invariant \(St\),
and suggests the existence of higher-degree local operations associated
with the same dual-complex structure.
\end{abstract}
\maketitle

\section{Introduction}\label{sec:Intro}

The purpose of this paper is to develop a discrete de Rham framework
for Arnold-type invariants and associated distribution-type structures
of immersed plane curves and immersed surfaces.
Within this framework, classical identities are naturally reinterpreted
through local finite-difference structures on dual complexes,
while the surface theory appears as a higher-dimensional extension
of the curve case.

The present paper does not pursue an abstract general theory alone.
One reason is that explicit Shumakovitch-type identities for surface immersions
have not previously been formulated.
Another reason is that the parallel development of the curve and surface cases
already exhibits a sufficiently rich correspondence within the discrete de Rham framework.
Moreover, particular emphasis is placed on the compatibility with the slice formula
and motion-picture viewpoint developed in \cite{ItoMizuno2026sliceF}.

We introduce what we call a \emph{discrete de Rham framework},
in which invariants are generated from local index data
via Stokes-compatible local operations on dual complexes.
Let $\phi$ denote the Alexander numbering on regions.
Viewing $\phi$ as a function on dual \(0\)-cells,
we show that fundamental invariants of plane curves and surface immersions
arise from explicitly defined higher-order local operations applied to \(\phi\).
In particular, the invariants \(St_{(1)}\) and \(St_{(2)}\)
are realized, after suitable normalizations,
as evaluations of \(d^2\phi\) and \(d^3\phi\), respectively.
Our first main result establishes normalized discrete
Stokes-type compatibilities on dual complexes
(Theorem~\ref{thm:DS}),
showing that evaluations on adjacent dual skeleta
are related through local finite-difference structures.

This viewpoint naturally explains classical identities,
for instance the Shumakovitch identity \cite{Shumakovich1995}
\[
St(C)
=
\sum_{p\in\Delta_0(C)} \varepsilon(p)\,\ind(p)
=
\frac12 \sum_{e\in\Delta_1(C)} s(e)\,\ind(e)^2
=
\frac13 \sum_{R\in\Delta_2(C)} g(R)\,\ind(R)^3,
\]
in the normalization where the base point is chosen
on an edge adjacent to the unbounded region.
Moreover, our second main result gives corresponding
surface identities for \(St_{(2)}\)
(Theorem~\ref{thm:ShumakovitchF}),
where evaluations of \(d^2\phi\) and \(d^3\phi\)
on dual skeleta recover triple-point contributions
through normalized local finite-difference relations.
These identities are understood as consequences
of discrete Stokes-type compatibilities.
In this sense, the framework developed here unifies
curve and surface identities through a common evaluation structure
on dual complexes.

A key feature of the present framework is that
it does not require a full discrete de Rham complex structure.
Instead, the essential structure is governed by
highest-degree local Stokes compatibilities on dual cells.
For curves, the compatibility between \(d^2\phi\) and \(d^1\phi\)
admits a complete local closure around dual \(2\)-cells.
For surfaces, the highest-degree compatibility between
\(d^3\phi\) and \(d^2\phi\)
admits a symmetric cubical realization,
whereas lower-degree compatibilities do not form a full cubical closure.
This distinction explains why the untwisted invariants
\(St_{(1)}\) and \(St_{(2)}\)
fit naturally into the local finite-difference framework.

Another key geometric ingredient of our approach is the slice interpretation
of surface immersions.
We reinterpret the slice formula of \cite{ItoMizuno2026sliceF}
\[
dSt_{(2)} = dSt_{(1)} + \sgn(\underline{\mathbb{R}^2})
\]
as an expression of discrete derivatives along a cooriented height function.
In this formulation, variations of invariants are encoded
by local differences of functions on the dual complex,
and slicing corresponds to a discrete integration along the height direction.
In contrast, the original Arnold invariant \(St(C)\)
requires an additional global twisting structure:
\[
St(C)=\sum_{d\in D(C)} s(d)\,\ind(d),
\]
where the sign system \(s(d)\) depends globally on the curve
and cannot be recovered solely from local index data.
Thus \(St\) should not be regarded as a failure of the present framework,
but rather as a globally twisted extension of the untwisted local hierarchy.
From this viewpoint,
\(St\) occupies the position of a twisted version of \(St_{(1)}\),
associated with a nontrivial global sign system.
The present work therefore clarifies the distinction between
untwisted local closures governed purely by finite differences
and dual incidences,
and globally twisted structures requiring additional sign transport data.
In this perspective,
\(St_{(2)}\) may also be regarded as a higher-dimensional analogue
of \(St_{(1)}\).

\noindent
\textbf{Structure of the paper.}
Section~\ref{sec:prelim} reviews generic immersions,
Alexander numberings, and induced cell decompositions.
Section~\ref{sec:dual} introduces dual complexes,
the discrete de Rham framework,
and the locally normalized finite-difference structures
for curves and surfaces.
Section~\ref{sec:ShId} develops explicit sign realizations
and Shumakovitch-type identities through local finite-difference computations,
and reformulates these identities conceptually within the discrete de Rham framework.
Section~\ref{sec:slice-drum} interprets the slice formula
within this framework and explains the compatibility between slicing
and the discrete local structures.
Finally, Section~\ref{sec:outlook} discusses higher-degree structures
and possible extensions of the present framework.

\section{Preliminaries}\label{sec:prelim}

\subsection{Generic immersions and local moves}

\begin{definition}[generic immersion]
A smooth immersion of a closed oriented surface into $\mathbb{R}^3$
is said to be \emph{generic} if its image exhibits only transverse
double intersections and isolated triple intersections.
\end{definition}

We identify an immersion with its image when no confusion arises.

\begin{definition}[local singularity jumps]
Two generic immersions are said to be \emph{generically regularly homotopic}
if they are related by diffeomorphisms and finitely many local moves
of the following four types:
\[
\textup{E},\quad \textup{H},\quad \textup{T},\quad \textup{Q}.
\]
\end{definition}

\begin{itemize}
\item[(E)] elliptic tangency of two sheets,
\item[(H)] hyperbolic tangency of two sheets,
\item[(T)] tangency of a double line to a sheet,
\item[(Q)] quadruple point (four sheets intersecting).
\end{itemize}

These are precisely the codimension-one singular events
in the sense of \cite{Nowik2004}, based on the classification of
Goryunov \cite{Goryunov1997}.

\medskip

Fix an orientation of $\Sigma$.
Then the immersion inherits a natural coorientation in $\mathbb{R}^3$,
which allows a refinement of these moves into types
\[
\textup{E}^j,\quad \textup{H}^j,\quad \textup{T}^j,\quad \textup{Q}^j,
\]
according to the number of positively cooriented sheets
(see also \cites{Goryunov1997,Nowik2004}).

\subsection{Cell decomposition}

Let $S:\Sigma \looparrowright \mathbb{R}^3$ be a generic immersion.
Then $\mathbb{R}^3$ admits a natural cell decomposition:

\begin{itemize}
\item $3$-cells (regions): components of $\mathbb{R}^3 \setminus S(\Sigma)$,
\item $2$-cells (sheets),
\item $1$-cells (double lines),
\item $0$-cells (triple points).
\end{itemize}

We denote by $\Delta_k(S)$ the set of $k$-cells.

\subsection{Alexander numbering}

\begin{definition}[Alexander numbering]
Using the coorientation, we assign values to regions as follows:
fix the unbounded region with value $-\tfrac{3}{2}$,
and increase or decrease the value when crossing sheets
according to the coorientation.

This defines a function
\[
\ind : \{\text{regions}\} \to \mathbb{Z}+\tfrac{1}{2},
\]
called the \emph{Alexander numbering}.
\end{definition}

\begin{remark}
The Alexander numbering is globally well-defined on all regions,
once a base region is fixed.
\end{remark}

\begin{definition}[Alexander numbering for plane curves]
Let $C\subset \mathbb{R}^2$ be an oriented generic immersed plane curve.
We choose the coorientation of $C$ to be the left normal direction
with respect to the orientation of the curve.

We assign the Alexander numbering to the regions of
$\mathbb{R}^2\setminus C$ by declaring the unbounded region to have value $0$,
and by increasing the value by $1$ whenever one crosses the curve
in the coorientation direction.

This defines a function
\[
\ind:\{\text{regions of } \mathbb{R}^2\setminus C\}\to \mathbb{Z},
\]
called the Alexander numbering of the plane curve.
\end{definition}

\begin{remark}
Equivalently, after smoothing the curve, the Alexander numbering of a region
may be interpreted as the signed number of times the region is enclosed.
For example, it can be computed by drawing a generic ray from the region
to the unbounded region and counting crossings with signs
``right-to-left minus left-to-right'' with respect to the oriented curve.
\end{remark}

\subsection{Indices on cells}

\begin{definition}[indices]\label{def:indix}

\medskip
\noindent
{\bf Surface case.}
Let $T(S)$ denote the set of triple points.
\begin{itemize}
\item 
For a triple point $t\in T(S)$, or equivalently for a $0$-cell
$\sigma_0\in \Delta_0(S)$, define
\[
\ind(\sigma_0)
:= \text{the average of the Alexander numberings of the adjacent regions}.
\]
\item Similarly, indices $\ind(\sigma_1)$ and $\ind(\sigma_2)$
for double lines $\sigma_1$ and sheets $\sigma_2$
are defined by averaging the indices of adjacent regions.
\end{itemize}

\paragraph{\bf Curve case.}
Let $D(C)$ denote the set of double points of a generic immersed plane curve.
\begin{itemize}
\item For a double point $p\in D(C)$, or equivalently for a
$0$-cell $\sigma_0\in\Delta_0(C)$, define
\[
\ind(\sigma_0)
:= \text{the average of the Alexander numberings of the adjacent regions}.
\]
\item Similarly, the index $\ind(\sigma_1)$ of an arc $\sigma_1$
is defined by averaging the Alexander numberings of the two adjacent regions.
\end{itemize}

\end{definition}

\begin{remark}
These averages are well-defined since adjacent regions differ by integer values.
\end{remark}

\section{Dual complex and the discrete de Rham framework}\label{sec:dual}

\subsection{Evaluation and discrete Stokes principle}\label{sec:principle}

To extract numerical invariants from locally defined data on the dual complex,
we consider the natural evaluation of functions on dual cells.

For a function $\psi$ defined on dual $k$-cells and a dual $k$-cell $X_k$,
we write
\[
\langle \psi, X_k \rangle
\]
for its evaluation on $X_k$.

We denote by $\Delta_k^\vee(S)$ and $\Delta_k^\vee(C)$
the sets of dual $k$-cells for surfaces and curves, respectively.

\begin{definition}
Let $\phi_S$ be the function assigning to each region
of a surface immersion its Alexander numbering.
We regard $\phi_S$ as a function on $\Delta_0^\vee(S)$:
\[
\phi_S \in \mathcal{K}^0(S)
:=
\{f:\Delta_0^\vee(S)\to\mathbb{Q}\}.
\]

Similarly, let $\phi_C$ denote the Alexander numbering
for a plane curve.
We regard $\phi_C$ as a function on $\Delta_0^\vee(C)$:
\[
\phi_C \in \mathcal{K}^0(C)
:=
\{f:\Delta_0^\vee(C)\to\mathbb{Q}\}.
\]
\end{definition}

Let
\[
d^k:\mathcal K^0(S)\to\mathcal K^k(S)
\]
for $k=1,2,3$, and
\[
d^k:\mathcal K^0(C)\to\mathcal K^k(C)
\]
for $k=1,2$,
denote maps defined by explicit local formulas.

The notation $d^k\phi$ refers to the locally defined value
of the map $d^k$ applied to $\phi$,
not to an iterated differential operator.

To describe compatibility relations, we also use boundary operators
\[
\partial:\Delta_k^\vee(S)\to\Delta_{k-1}^\vee(S),
\qquad
\partial:\Delta_k^\vee(C)\to\Delta_{k-1}^\vee(C),
\]
which assign to each dual $k$-cell the formal sum
(with orientation signs included)
of its incident dual $(k-1)$-cells.

\begin{theorem}[Discrete Stokes compatibility]\label{thm:DS}

The explicitly defined maps
\[
d^k:\mathcal K^0(S)\to\mathcal K^k(S),
\qquad
d^k:\mathcal K^0(C)\to\mathcal K^k(C),
\]
together with the evaluation pairings
\[
\langle -, - \rangle,
\]
satisfy the following normalized local compatibilities.

For surfaces, one has
\begin{equation}\label{eq:DSsurfaceJ}
\frac1{24}
\langle d^3\phi_S,X_3\rangle
=
\frac1{18}
\langle d^2\phi_S,\partial X_3\rangle
\end{equation}
on symmetric cubical configurations surrounding triple points.

For curves, one has
\begin{equation}\label{eq:DScurveJ}
\frac1{3!}
\langle d^2\phi_C,X_2\rangle
=
\frac1{2!}
\langle d^1\phi_C,\partial X_2\rangle
\end{equation}
around double points.

\end{theorem}

\begin{remark}
The above relations describe highest-degree local Stokes compatibilities
for the local configurations surrounding double and triple points.
They are not derived from iterated cochain operators,
and no full cochain-complex structure is assumed.

In the curve case, the compatibility between \(d^2\phi_C\)
and \(d^1\phi_C\) closes completely around dual \(2\)-cells.
In the surface case, the compatibility between
\(d^3\phi_S\) and \(d^2\phi_S\)
admits a symmetric cubical realization at the highest degree,
whereas lower-degree compatibilities do not form a complete cubical closure.
\end{remark}

Theorem~\ref{thm:DS} is proved in
Sections~\ref{subsec:CellDecomp}--\ref{subsec:DScurve}.

\subsection{Cell decomposition induced by an immersion}\label{subsec:CellDecomp}

Let $S:\Sigma \looparrowright \mathbb{R}^3$ be a generic immersion.
Then $\mathbb{R}^3$ admits a cell decomposition:
\begin{itemize}
\item $3$-cells (regions),
\item $2$-cells (sheets),
\item $1$-cells (double lines),
\item $0$-cells (triple points).
\end{itemize}

We denote this decomposition by $\Delta(S)$.

\subsection{Dual complex for surfaces}

Let $\Delta^\vee(S)$ be the dual complex of $\Delta(S)$.
Each dual $k$-cell corresponds to a $(3-k)$-cell of $\Delta(S)$.

Thus functions on $\Delta_k^\vee(S)$ encode data on $(3-k)$-cells.

Definition~\ref{def:indix} induces a propagation of values
from regions to higher-dimensional cells.

More precisely, for a $0$-cell $\sigma_0\in\Delta_0(S)$
and a $k$-cell $\sigma_k\in\Delta_k(S)$ adjacent to it,
we assign
\[
\ind(\sigma_0)\longmapsto \ind(\sigma_k).
\]

\subsection{Discrete differential structure for surfaces}\label{sec:Discrete}

The definitions of the maps \(d^k\) are given as follows. 
The signs \(\sgn(\cdot)\) appearing below are determined by local
orientations and coorientations induced from the Alexander numbering
(see Section~\ref{sec:ShStTwo} for explicit conventions).

\begin{itemize}

\item \(d^1\phi_S\):
\[
d^1 \phi_S (\sigma^{\vee}_1)
=
\sum_{\sigma^{\vee}_2 \sim \sigma^{\vee}_1}
\sgn(\sigma^{\vee}_2)\,\ind(\sigma^{\vee}_2)^2 ,
\]
where the sum runs over $\sigma^{\vee}_2$
adjacent to $\sigma^{\vee}_1$.

\item \(d^2\phi_S\):
\[
d^2 \phi_S (\sigma^{\vee}_2)
=
\sum_{\sigma^{\vee}_1 \sim \sigma^{\vee}_2}
\sgn(\sigma^{\vee}_1)\,\ind(\sigma^{\vee}_1)^3 ,
\]
where the sum runs over $\sigma^{\vee}_1$
adjacent to $\sigma^{\vee}_2$.

\item \(d^3\phi_S\):
\[
d^3 \phi_S (\sigma^{\vee}_3)
=
\sum_{\sigma^{\vee}_0 \sim \sigma^{\vee}_3}
\sgn(\sigma^{\vee}_0)\,\ind(\sigma^{\vee}_0)^4 ,
\]
where the sum runs over $\sigma^{\vee}_0$
adjacent to $\sigma^{\vee}_3$.

\end{itemize}

\begin{proof}[Proof of Theorem~\ref{thm:DS} for surfaces]

\begin{figure}[htbp]
\centering
\includegraphics[width=0.38\linewidth]{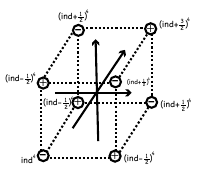}
\caption{Local cubical configuration for \(d^3\phi_S\) around a triple point.
The cube represents the dual \(3\)-cell surrounding the triple point.
Its vertices correspond to the adjacent dual \(0\)-cells, and the local
orientations determine the signed fourth-order finite-difference pattern.}
\label{fig:Dloop}
\end{figure}

\medskip
\noindent
{\bf Evaluation of the left-hand side of \eqref{eq:DSsurfaceJ}.}

Figure~\ref{fig:Dloop} illustrates the local contribution associated with
\(d^3\phi_S\).

The vertices of the cube indicate the adjacent dual \(0\)-cells,
and the alternating signs represent the local finite-difference pattern.

For the dual \(3\)-cell \(X_3\) surrounding a triple point,
the signed sum over the adjacent dual \(0\)-cells becomes
\[
(x+\tfrac32)^4
-3(x+\tfrac12)^4
+3(x-\tfrac12)^4
-(x-\tfrac32)^4
=
24x.
\]

After normalization by \(1/24\),
the local contribution reduces to
\[
x=\ind(t).
\]

\begin{figure}[htbp]
\centering
\includegraphics[width=0.42\linewidth]{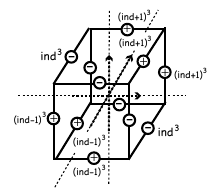}
\caption{Grouped local dual \(1\)-cell contributions associated with \(d^2\phi_S\).
The indicated segments represent local dual \(1\)-cell contributions
corresponding to sheet pieces adjacent to the boundary dual
\(2\)-cells of the cubical dual \(3\)-cell.}
\label{fig:DloopO}
\end{figure}

\medskip
\noindent
{\bf Evaluation of the right-hand side of \eqref{eq:DSsurfaceJ}.}

Figure~\ref{fig:DloopO} illustrates the grouped local contributions
associated with \(d^2\phi_S\).

The contributions are collected according to the three double-line
directions of the triple-point configuration.

With the local signs determined by the orientation and coorientation,
one obtains
\[
3(x+1)^3
-6x^3
+3(x-1)^3
=
18x.
\]

After normalization by \(1/18\),
the local contribution reduces to
\[
x=\ind(t).
\]

In Figure~\ref{fig:DloopO}, the dashed segments indicate the
double-line directions corresponding to dual \(1\)-cell contributions,
while the solid segments represent portions of the boundary
of the dual \(3\)-cell.

Therefore both sides of \eqref{eq:DSsurfaceJ}
reduce to the same local contribution \(\ind(t)\).
\end{proof}

\subsection{Dual complex for curves}

We next prepare notation for proving \eqref{eq:DScurveJ},
and then give the proof of
Theorem~\ref{thm:DS} for curves.

For a generic immersed plane curve
$C \subset \mathbb{R}^2$,
we proceed similarly.

In this case,
we use the Alexander numbering on the regions
of $\mathbb{R}^2\setminus C$
as defined in Section~\ref{sec:prelim},
and construct the dual complex accordingly.

The induced decomposition consists of:
\begin{itemize}
\item regions (2-cells),
\item arcs (1-cells),
\item double points (0-cells).
\end{itemize}

Let $\Delta^\vee(C)$ denote the dual complex.

Each dual $k$-cell corresponds to a $(2-k)$-cell
of the original cell decomposition.

In particular,
a dual $2$-cell corresponding to a double point
is a small disk neighborhood in the dual complex,
whose boundary consists of four dual $1$-cells.

We define
\[
\phi_C \in \mathcal K^0(C)
\]
to be the Alexander numbering on regions.

\subsection{Discrete differential structure for curves}\label{subsec:DScurve}

We now give the explicit local definitions
of the maps \(d^k\) for plane curves.

Let $C \subset \mathbb{R}^2$
be a generic immersed curve,
and let
\[
\phi_C \in \mathcal K^0(C)
\]
be the Alexander numbering on regions.

The definitions of the maps $d^k$ are given as follows.  
The signs \(\sgn(\cdot)\) appearing below are determined by local
orientations and coorientations induced from the Alexander numbering
(see Section~\ref{sec:ShStOne} for explicit conventions).

\begin{itemize}

\item \(d^1\phi_C\):
\[
d^1 \phi_C (\sigma^{\vee}_1)
=
\sum_{\sigma_1 \sim \sigma^{\vee}_1}
\sgn(\sigma_1)\,\ind(\sigma_1)^2 ,
\]
where the sum runs over the relevant arc segments
incident to the dual \(1\)-cell \(\sigma^\vee_1\).

\item \(d^2\phi_C\):
\[
d^2 \phi_C (\sigma^{\vee}_2)
=
\sum_{\sigma^{\vee}_0 \sim \sigma^{\vee}_2}
\sgn(\sigma^{\vee}_0)\,\ind(\sigma^{\vee}_0)^3 ,
\]
where the sum runs over the dual \(0\)-cells
corresponding to the regions adjacent to the double point
represented by \(\sigma^\vee_2\).

\end{itemize}

\begin{proof}[Proof of Theorem~\ref{thm:DS} for curves]

\begin{figure}[htbp]
\centering
\includegraphics[width=0.45\linewidth]{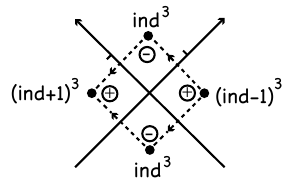}
\caption{Local Stokes configuration for \(d^2\phi_C\) around a double point.
The dotted loop represents the boundary \(\partial\sigma^\vee_2\)
of the dual \(2\)-cell,
and the local orientations determine the signed finite-difference pattern.}
\label{fig:loop}
\end{figure}

\medskip
\noindent
{\bf Evaluation of the left-hand side of \eqref{eq:DScurveJ}.}

Figure~\ref{fig:loop}
illustrates the local Stokes-type relation around a double point.

The dual \(2\)-cell \(\sigma^\vee_2\)
is a small disk surrounding the double point,
and its boundary \(\partial\sigma^\vee_2\)
is the oriented loop formed by four dual \(1\)-cells.

Evaluating \(d^2\phi_C\)
on this local configuration yields
\[
(i+1)^3-i^3-i^3+(i-1)^3
=
6i,
\]
which is precisely the second finite difference
\[
D^2(x^3)=3!\,x.
\]

After normalization by \(1/3!\),
the local contribution reduces to
\[
i=\ind(p).
\]

\begin{figure}[htbp]
\centering
\includegraphics[width=0.35\linewidth]{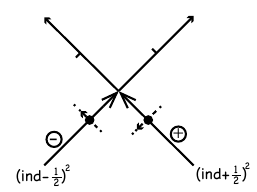}
\caption{Local directional contribution associated with \(d^1\phi_C\).
Only the two incoming local edge contributions are counted
in the evaluation of \(d^1\phi_C\).}
\label{fig:injective}
\end{figure}

\medskip
\noindent
{\bf Evaluation of the right-hand side of \eqref{eq:DScurveJ}.}

Figure~\ref{fig:injective}
illustrates the local directional contribution associated with \(d^1\phi_C\).
Readers should note that only incoming edge contributions
are counted here in order to avoid double counting.

With the local signs determined by the orientation and coorientation,
one obtains
\[
(i+\tfrac12)^2-(i-\tfrac12)^2
=
2i,
\]
which corresponds to the first finite difference
\[
D(x^2)=2x.
\]

After normalization by \(1/2!\),
the local contribution reduces to
\[
i=\ind(p).
\]

Therefore both sides of \eqref{eq:DScurveJ}
reduce to the same local contribution \(\ind(p)\).
\end{proof}

  \section{Local sign realization and explicit Shumakovitch identities}\label{sec:ShId}

\subsection{Shumakovitch-type identities}\label{subsec:Shumakovitch}

A key point is that the present framework is governed by
highest-degree local closure phenomena on dual skeleta.
More precisely, the \(0\)-cell contribution is recovered from higher-cell data
through normalized local finite-difference identities
compatible with discrete Stokes-type relations.

For curves, the highest-degree local compatibility
closes completely around dual \(2\)-cells.
For surfaces, the highest-degree structure admits a symmetric cubical realization,
although lower-degree compatibilities do not form a full cubical closure.

\subsubsection{The original Shumakovitch identity for plane curves}

We record a concrete sign/gleam convention for the original Shumakovitch identity
for an oriented generic immersed plane curve.

\paragraph{Double points (0-cells).}
Let $C\subset \mathbb{R}^2$ be an oriented generic immersion,
and let $p$ be a double point.
Assume that the base point is chosen on an edge adjacent
to the unbounded region.

We define the local sign $\varepsilon(p)\in\{\pm1\}$
to be the sign of the intersection form,
with the following convention:
\begin{itemize}
\item $\varepsilon(p)=-1$ if the \emph{lower-left} branch passes first,
\item $\varepsilon(p)=+1$ if the \emph{lower-right} branch passes first.
\end{itemize}

\paragraph{Incident edges (1-cells).}
Let $(e_L,e_R)$ be the two incident local branches at $p$.
We assign edge signs $(s_L,s_R)$ by
\[
(s_L,s_R)=(-,+)\quad\text{if }\varepsilon(p)=-1,
\qquad
(s_L,s_R)=(+,-)\quad\text{if }\varepsilon(p)=+1.
\]

\paragraph{Adjacent regions and gleam weights (2-cells).}
Let $R_1,R_2,R_3,R_4$ be the four local regions adjacent to $p$.
Pair $R_2$ and $R_4$ among $R_1,\dots,R_4$ so that they have the same Alexander numbering.

Then we assign local gleam weights $g_p(R)\in\{\pm\frac12\}$ by
\begin{align*}
g_p(R)&=
\begin{cases}
-\frac12 & \text{for $R_1, R_3$,}\\[2pt]
+\frac12 & \text{for $R_2, R_4$,}
\end{cases}
\quad\text{if }\varepsilon(p)=-1,
\\
&\text{If $\varepsilon(p)=+1$, we reverse all the above local gleam signs.}
\end{align*}

The local gleam weights extend to a global function on regions by
\[
g(R) = \sum_{p \in \partial R} g_p(R),
\]
where the sum runs over all double points adjacent to $R$.

Turaev \cite{Turaev1994Book} introduced the notion of the gleam $g(R)$,
which is globally well-defined.
Using this notion, Shumakovitch \cite{Shumakovich1995} deduced his identity
(Fact~\ref{fact:Sh}).

With these conventions, the following equalities hold.

\begin{fact}[Original Shumakovitch identity]\label{fact:Sh}
\begin{equation}\label{eq:Shumakovitch-curve-original}
St(C)
=
\sum_{p\in\Delta_0(C)} \varepsilon(p)\,\ind(p)
=
\frac12 \sum_{e\in\Delta_1(C)} s(e)\,\ind(e)^2
=
\frac13 \sum_{R\in\Delta_2(C)} g(R)\,\ind(R)^3,
\end{equation}
where $\ind(\cdot)$ denotes the Alexander numbering
and $s(e)$ and $g(R)$ are determined by the above edge-sign and gleam conventions.
\end{fact}

\paragraph{Elementary verification at one double point.}
Fix a double point $p$ and write
\[
\varepsilon=\varepsilon(p),
\qquad
i=\ind(p).
\]

\smallskip
\noindent\emph{(0-cell level).}
The contribution of $p$ is $\varepsilon\, i$.

\smallskip
\noindent\emph{(1-cell level).}
The two incident edges have indices $i\pm \tfrac12$ with signs $\pm\varepsilon$ as above, hence
\[
\frac12\Bigl[
\varepsilon\,(i+\tfrac12)^2
-
\varepsilon\,(i-\tfrac12)^2
\Bigr]
=
\varepsilon\, i.
\]

\smallskip
\noindent\emph{(2-cell level).}
Among the four adjacent regions, two have index $i+1$ and $i-1$,
and the other two have index $i$, with the alternating pattern encoded by the gleam weights.

Thus the local region-level contribution equals
\[
\frac13\cdot \frac{\varepsilon}{2}
\Bigl[
(i+1)^3+(i-1)^3-2i^3
\Bigr].
\]

Since
\[
(i+1)^3+(i-1)^3-2i^3=6i,
\]
we obtain
\[
\frac13\cdot \frac{\varepsilon}{2}\cdot 6i
=
\varepsilon\, i.
\]

This proves the equality of the three expressions in
\eqref{eq:Shumakovitch-curve-original}.

\begin{remark}
The above comparison suggests that the invariant $St$
may be viewed as a twisted version of $St_{(1)}$,
associated with a nontrivial global sign structure.
\end{remark}

\subsubsection{Untwisted plane-curve identity for $St_{(1)}$}

We next record the corresponding untwisted identity for $St_{(1)}$.
Unlike the original Arnold invariant $St$, this invariant does not use
the global sign system of double points.
In particular, the region-level formula below uses local signs,
not Turaev gleams.

Let $p$ be a double point of $C$, and write
\[
i=\ind(p).
\]

In a small neighborhood of $p$, the four adjacent regions have Alexander
numberings
\[
i+1,\quad i,\quad i,\quad i-1,
\]
after relabeling if necessary.

The index of an incident edge is defined
as the average of the Alexander numberings of the two adjacent regions,
and the index of the double point is the average of the four adjacent
region numberings, namely $i=\ind(p)$.

\paragraph{Edge level.}
Among the edge germs incident to $p$, consider the two germs directed
toward $p$.
Their indices are
\[
i+\frac12,\qquad i-\frac12.
\]

We define the local sign by
\[
s_p(\sigma_1)=+1
\quad\text{if}\quad
\ind(\sigma_1)=i+\frac12,
\]
and
\[
s_p(\sigma_1)=-1
\quad\text{if}\quad
\ind(\sigma_1)=i-\frac12.
\]

Then the local edge contribution is
\[
\bigl(i+\tfrac12\bigr)^2-\bigl(i-\tfrac12\bigr)^2
=
2i.
\]

Hence the contribution of $p$ to $St_{(1)}$
is obtained by multiplying the local edge sum by $\frac12$.

\begin{remark}
If all four incident edge germs are counted,
then incoming and outgoing germs produce the same local contribution twice.
This introduces a multiplicity $m=2$,
and the same normalized formula is recovered after dividing by this multiplicity.
\end{remark}

\paragraph{Region level.}
The index of each region is its Alexander numbering itself.

We define the local sign on the four adjacent regions as follows:
the two regions for which the two incident boundary directions are both
entering or both leaving are assigned sign $-1$,
and the other two regions are assigned sign $+1$.

With this convention, the local region contribution at $p$ is
\[
-2i^3+(i+1)^3+(i-1)^3
=
6i.
\]

Hence the contribution of $p$ to $St_{(1)}$
is obtained by multiplying the local region sum by $\frac16$.

\begin{remark}
At the region level, the four adjacent regions are counted without
additional multiplicity for the fixed double point.
Thus the effective multiplicity is $m=1$
in this local computation.
\end{remark}

Summing these local identities over all double points gives
\[
St_{(1)}(C)
=
\sum_{p\in\Delta_0(C)} \ind(p)
=
\frac12
\sum_{(p,\sigma_1)} s_p(\sigma_1)\,\ind(\sigma_1)^2
=
\frac16
\sum_{(p,\sigma_2)} s_p(\sigma_2)\,\ind(\sigma_2)^3,
\]
where the sums are taken over the corresponding local incidences
at double points.

\subsubsection{Untwisted surface identity for $St_{(2)}$}

Let $S:\Sigma \looparrowright \mathbb{R}^3$
be a generic oriented and cooriented immersion
and let $t$ be a triple point.

We define signs on the cells in a small ball neighborhood of $t$.

\paragraph{Triple points (0-cells).}
We set
\[
s(t)=+1.
\]

\paragraph{Double lines (1-cells).}
Each incident double line is oriented in the direction along which the Alexander index increases.

For a local double-line segment $e$ incident to $t$, we define
\[
s_t(e)=
\begin{cases}
-1 & \text{if $e$ is oriented toward $t$,}\\
+1 & \text{if $e$ is oriented away from $t$.}
\end{cases}
\]

\paragraph{Sheets (2-cells).}
For a sheet piece $f$ incident to $t$,
let $e_1,e_2$ be the two incident double-line segments
in $\partial f$ meeting at $t$.

We define
\[
s_t(f)=
\begin{cases}
+1 & \text{if $e_1$ and $e_2$ have the same orientation,}\\
-1 & \text{otherwise.}
\end{cases}
\]

\paragraph{Regions (3-cells).}
For a region $R$ incident to $t$,
let $f_1,f_2,f_3$ be the three sheet pieces bounding $R$
near $t$.

We define
\[
s_t(R)=s_t(f_1)s_t(f_2)s_t(f_3).
\]

This definition is induced from the orientation
of the incident double-line segments.

\begin{theorem}[Surface Shumakovitch-type identity]\label{thm:ShumakovitchF}
\begin{equation}\label{eq:Shumakovitch-surface}
\begin{split}
\StTwo(S)
= \sum_{t \in\Delta_0(S)}\ind(t)
&=
\frac{1}{6}\sum_{(t,\sigma_1)} s_t(\sigma_1)\,\ind(\sigma_1)^2 \\
&=
\frac{1}{18}\sum_{(t,\sigma_2)} s_t(\sigma_2)\,\ind(\sigma_2)^3 \\
&=
\frac{1}{24}\sum_{(t,\sigma_3)} s_t(\sigma_3)\,\ind(\sigma_3)^4,
\end{split}
\end{equation}
where the sums are taken over local incidences between triple points
and the corresponding adjacent cells.
\end{theorem}

\begin{proof}
Fix a triple point $t$.
The following computation verifies the contribution of the local
incidences adjacent to $t$.

\smallskip
\noindent\emph{(Edge level).}
The three incident double lines contribute symmetrically, hence
\[
\frac{1}{6}\Bigl[
3\bigl(\ind(t)+\tfrac{1}{2}\bigr)^2
-3\bigl(\ind(t)-\tfrac{1}{2}\bigr)^2
\Bigr]
=
\ind(t).
\]

\smallskip
\noindent\emph{(Face level).}
The induced signs on the twelve sheet pieces yield the alternating sum
\[
\frac{1}{18}\Bigl[
3\bigl(\ind(t)+1\bigr)^3
-6\ind(t)^3
+3\bigl(\ind(t)-1\bigr)^3
\Bigr]
=
\ind(t).
\]

\smallskip
\noindent\emph{(Region level).}
The eight local regions yield the third finite-difference pattern
\[
\frac{1}{24}\Bigl[
\bigl(\ind(t)+\tfrac{3}{2}\bigr)^4
-3\bigl(\ind(t)+\tfrac{1}{2}\bigr)^4
+3\bigl(\ind(t)-\tfrac{1}{2}\bigr)^4
-\bigl(\ind(t)-\tfrac{3}{2}\bigr)^4
\Bigr]
=
\ind(t).
\]
\end{proof}

The above elementary argument admits a reformulation
in terms of the discrete Stokes-type framework explained in
Sections~\ref{sec:ShSt}--\ref{sec:ShStTwo}.

\subsubsection{The original Shumakovitch identity via local finite differences}\label{sec:ShSt}

Let $C\subset \mathbb{R}^2$ be an oriented generic immersed plane curve.
Assume that the base point is chosen on an edge adjacent to the unbounded region.

\paragraph{Dual complex.}
Let $\Delta(C)$ be the cell decomposition of $\mathbb{R}^2$
induced by $C$, and let $\Delta^\vee(C)$ be its barycentric dual.

Thus:
\begin{itemize}
\item dual $0$-cells correspond to regions,
\item dual $1$-cells cross arc segments,
\item each double point corresponds to a small dual $2$-cell.
\end{itemize}

\paragraph{Alexander numbering as a function on dual $0$-cells.}
Let
\[
\phi \in \mathcal{K}^0(C)
\]
be the Alexander numbering on dual $0$-cells:
\[
\phi(R^\vee)=\ind(R).
\]

\paragraph{Edge level.}
Let $p$ be a double point and write
\[
i=\ind(p).
\]

Among the local edge germs incident to $p$,
we consider the two germs directed toward $p$.

Their indices are
\[
i+\frac12,\qquad i-\frac12.
\]

The local finite-difference contribution is
\[
\bigl(i+\tfrac12\bigr)^2-\bigl(i-\tfrac12\bigr)^2
=
2i.
\]

For the original Arnold invariant,
this local edge contribution is twisted by the global sign
$\varepsilon(p)$ determined by the chosen base point
and the orientation convention.

Thus the edge-level contribution becomes
\[
\frac12\,
\varepsilon(p)
\Bigl[
\bigl(i+\tfrac12\bigr)^2-\bigl(i-\tfrac12\bigr)^2
\Bigr]
=
\varepsilon(p)i.
\]

Equivalently, this gives the edge-level term
\[
\frac12 \sum_{e\in\Delta_1(C)} s(e)\,\ind(e)^2
\]
in the original Shumakovitch identity.

\paragraph{Region level with gleam.}
At the region level, the same twisting is encoded by the gleam weights.

Around the double point $p$,
the four adjacent regions have Alexander numberings
\[
i+1,\quad i,\quad i,\quad i-1
\]
after relabeling if necessary.

With the gleam convention described above,
the local region contribution is
\[
\frac13\cdot \frac{\varepsilon(p)}{2}
\Bigl[
(i+1)^3+(i-1)^3-2i^3
\Bigr].
\]

Since
\[
(i+1)^3+(i-1)^3-2i^3=6i,
\]
this contribution is equal to
\[
\varepsilon(p)i.
\]

Therefore one obtains
\[
St(C)
=
\sum_{p\in\Delta_0(C)} \varepsilon(p)\ind(p)
=
\frac12\sum_{e\in\Delta_1(C)} s(e)\ind(e)^2
=
\frac13\sum_{R\in\Delta_2(C)} g(R)\ind(R)^3.
\]

\paragraph{Effective-factor principle (curve case).}
The normalization coefficients arise from two ingredients:
the finite-difference coefficient and the effective multiplicity
of local contributions to the $0$-cell index.

For details, see Section~\ref{sec:Normalization}.

\subsubsection{The plane-curve case via local finite differences}\label{sec:ShStOne}

Let $C\subset \mathbb{R}^2$ be an oriented generic immersed curve,
and let $p$ be a double point.

\paragraph{Dual complex.}
In a small neighborhood of $p$:
\begin{itemize}
\item 2 local directions (smoothing directions),
\item 4 regions.
\end{itemize}

Let $\Delta^\vee(C)$ be the dual complex.

\paragraph{Alexander numbering.}
Let
\[
\phi \in \mathcal{K}^0(C)
\]
be the Alexander numbering on regions.

\paragraph{Edge level.}
Near a double point,
there are two local directions corresponding to
the two smoothing directions of the crossing.

One of these connects regions with the same Alexander numbering
and hence gives no contribution.

The other connects regions whose Alexander numbers differ by
\(\pm \tfrac12\), and yields
\[
(i+\tfrac12)^2-(i-\tfrac12)^2=2i.
\]

Thus the effective contribution is \(2i\),
yielding the normalization
\[
\frac{1}{2}.
\]

\paragraph{Region level.}
The four surrounding regions produce the second finite difference
\[
(i+1)^3+(i-1)^3-2i^3 = 6i,
\]
and hence the normalization coefficient is
\[
\frac{1}{6}.
\]

\medskip

Thus one obtains
\begin{align*}
\StOne (C)
&=
\sum_{p \in\Delta_0(C)} \ind(p) \\
&=
\frac{1}{2}
\sum_{(p,\sigma_1)}
s_p(\sigma_1)\ind^2(\sigma_1) \\
&=
\frac{1}{6}
\sum_{(p,\sigma_2)}
s_p(\sigma_2)\ind^3(\sigma_2),
\end{align*}
\noindent
where each sum is taken over local incidences at double points.

\subsubsection{The surface triple-point case via local finite differences}\label{sec:ShStTwo}

Let $S:\Sigma \looparrowright \mathbb{R}^3$ be a generic immersion,
and let $t$ be a triple point.

\paragraph{Dual complex.}
In a small ball neighborhood of $t$:
\begin{itemize}
\item 3 double lines,
\item 12 sheet pieces,
\item 8 regions.
\end{itemize}

Let $\Delta^\vee(S)$ be the dual complex.

\paragraph{Alexander numbering.}
Let
\[
\phi \in \mathcal{K}^0(S)
\]
be the Alexander numbering on regions.

\paragraph{Edge level.}
Near a triple point,
there are three independent local directions
corresponding to the three double lines.

Each double line connects regions whose Alexander numbers differ by
\(\pm \tfrac12\),
and hence produces the finite difference
\[
(x+\tfrac12)^2-(x-\tfrac12)^2=2x.
\]

Thus the total contribution is
\[
3\cdot 2x = 6x,
\]
yielding the normalization
\[
\frac{1}{6}.
\]

\paragraph{Face level.}
Each sheet contribution produces the second finite difference
\[
D^2(x^3)=3!\,x.
\]

There are three symmetric families of such local configurations.

Hence the total factor becomes
\[
3\cdot 3! = 18,
\]
yielding the normalization
\[
\frac{1}{18}.
\]

\paragraph{Region level.}
The eight surrounding regions produce the third finite difference
\[
D^3(x^4)=4!\,x,
\]
and hence the normalization coefficient is
\[
\frac{1}{24}.
\]

\medskip

Thus one obtains
\begin{align*}
\StTwo (S)
&=
\sum_{t \in\Delta_0(S)} \ind(t) \\
&=
\frac{1}{6}
\sum_{(t,\sigma_1)}
s_t(\sigma_1)\ind^2(\sigma_1) \\
&=
\frac{1}{18}
\sum_{(t,\sigma_2)}
s_t(\sigma_2)\ind^3(\sigma_2) \\
&=
\frac{1}{24}
\sum_{(t,\sigma_3)}
s_t(\sigma_3)\ind^4(\sigma_3),
\end{align*}
\noindent
where each sum is taken over local incidences at triple points.

\subsection{Normalization mechanism}\label{sec:Normalization}

The normalization of all identities in this paper
is governed by the finite-difference identity
\[
D^r(x^{r+1}) = (r+1)! \, x.
\]

More precisely,
two independent factors determine the coefficients:

\begin{itemize}
\item the finite-difference coefficient $(r+1)!$ arising from $D^r$,
\item the multiplicity $m$ of local contributions to the $0$-cell index.
\end{itemize}

Thus the normalization coefficient is
\[
\frac{1}{m(r+1)!}.
\]

\begin{example}
For surfaces,
one has $m=3$ at the levels of edges and faces,
and $m=1$ at the level of regions,
yielding
\[
\frac{1}{3 \cdot (1+1)!},\quad
\frac{1}{3 \cdot (2+1)!},\quad
\frac{1}{1\cdot(3+1)!}.
\]
\end{example}

\begin{example}
For curves,
one has $m=1$ at the levels of edges and regions,
yielding
\[
\frac{1}{1 \cdot (1+1)!},
\quad
\frac{1}{1 \cdot (2+1)!}.
\]
\end{example}

\begin{remark}
The multiplicity $m$ should be understood as the effective redundancy
with which local configurations contribute to the $0$-cell index.

In the curve case,
part of this multiplicity is absorbed into the gleam weights,
and therefore does not appear as a separate uniform parameter.

A systematic formulation of this normalization structure will be given in a subsequent work.
\end{remark}

Through the Shumakovitch-type identities above,
we obtain normalized highest-degree local Stokes compatibilities
for the locally defined expressions \(d^k\phi\).

These compatibilities organize the local finite-difference computations
on dual complexes
and explain why higher powers of the Alexander index collapse
to linear contributions.

In particular, invariants constructed from expressions of the form
\(d^{r+1}\phi\)
reduce, after summation,
to linear functions of \(\ind\).

The curve case admits a complete local closure around dual \(2\)-cells,
while the surface case exhibits a symmetric highest-degree cubical closure.
This distinction explains why the untwisted invariants
\(St_{(1)}\) and \(St_{(2)}\)
fit naturally into the local finite-difference framework,
whereas the original Arnold invariant \(St\)
requires an additional global twisting structure.

\section{Slice realization of the de Rham framework}
\label{sec:slice-drum}

In this section,
we reinterpret the slice construction
within the discrete de Rham framework developed in
Sections~\ref{sec:dual} and \ref{sec:ShId}.

\paragraph{Notation.}
In this section,
the expressions
\[
dSt_{(1)},\qquad dSt_{(2)}
\]
denote the jumps of the corresponding invariants
across codimension-one transitions,
in accordance with the classical slice formula.

They should be distinguished from the local operators
\[
d^k
\]
appearing in the discrete de Rham framework.

\subsection{Slicing and compatibility}

Let $S:\Sigma \looparrowright \mathbb{R}^3$
be a generic immersion
in Morse position with respect to a height function $h$,
and let
\[
A_u = S \cap \{h = u\}
\]
be the slice curve.

The Alexander numbering restricts to each slice
(after a suitable global constant shift if necessary),
yielding
\[
\phi_{A_u} \in \mathcal{K}^0(A_u).
\]

Accordingly,
up to a suitable constant shift,
the local discrete expressions satisfy the compatibility
\[
d^2 \phi_{A_u}
=
\text{slice of } d^3 \phi_S.
\]

However,
this compatibility fails at codimension-one transitions,
where a discrepancy appears between restriction and variation.

\begin{remark}
Geometrically,
slicing converts triple-point contributions on the surface
into double-point contributions on the curve,
with a controlled defect recording the direction of passage
through the singular level.
\end{remark}

\subsection{Slice formula}

\begin{theorem}[slice formula in de Rham form {\cite{ItoMizuno2026sliceF}}]
Let $S$ undergo a codimension-one transition.
Then
\[
dSt_{(2)} = dSt_{(1)} + \sgn(\underline{\mathbb{R}^2}).
\]
\end{theorem}

The invariant $St_{(2)}$
is given by the evaluation of $d^3\phi_S$
on dual $3$-cells,
while $St_{(1)}$
is given by the evaluation of $d^2\phi_{A_u}$
on dual $2$-cells of the slice.

The compatibility above identifies
$d^2\phi_{A_u}$
with the restriction of $d^3\phi_S$
up to a suitable constant shift.

At singular transitions,
the compatibility between slicing
and discrete local operations fails,
and a defect term appears.

This defect is determined by the coorientation
of the slice passage,
yielding the formula above.

\section{Outlook: higher-degree structures}\label{sec:outlook}
The present framework is not based on a full discrete de Rham complex. Instead, the essential structure is given by highest-degree local Stokes compatibilities on dual cells. For curves, the compatibility between \(d^2\phi\) and \(d^1\phi\) admits a complete local closure around dual \(2\)-cells. For surfaces, the highest-degree compatibility between \(d^3\phi\) and \(d^2\phi\) admits a symmetric cubical realization, while lower-degree compatibilities do not form a full cubical closure. This distinction clarifies why the untwisted invariants \(St_{(1)}\) and \(St_{(2)}\) fit naturally into the local finite-difference framework, whereas the original Arnold invariant \(St\) requires an additional global twisting structure.

The constructions developed in this paper exhibit a uniform structural pattern: invariants arise as evaluations of locally defined expressions of the form
\[
\langle d^{k+1}\phi,\ \Delta_{k+1}^\vee \rangle,
\]
governed by the finite-difference identity
\[
D^r(x^{r+1}) = (r+1)! \, x,
\]
together with highest-degree local Stokes compatibilities on dual skeleta. This finite-difference structure explains the collapse of higher-degree index expressions to linear contributions and underlies both the curve and surface cases treated here. From a conceptual viewpoint, the repeated finite-difference reductions appearing in this framework may be regarded as a discrete analogue of iterated integration. Indeed, the identities
\[
D(x^2)=2!x,\qquad
D^2(x^3)=3!x,\qquad
D^3(x^4)=4!x
\]
play a role analogous to repeated applications of the fundamental theorem of calculus, reducing higher-degree local data to lower-degree contributions. In this sense, the present framework suggests an iterated-integral-type structure on dual skeleta.

These observations strongly suggest the existence of higher-order local operations \(d^k\) extending the present framework to arbitrary degrees. Such operations are expected to be governed by the same finite-difference structures and incidence relations of dual complexes. A systematic formulation and the well-definedness of such higher-degree operators, including a precise treatment of multiplicities and possible obstruction phenomena in lower-degree compatibilities, will be developed in subsequent work.

Another important aspect of the present framework is the distinction between finite order and polynomial degree. The \emph{finite order} (Nowik) measures sensitivity to codimension-one events, whereas the \emph{polynomial (moment) degree} records dependence on powers of the Alexander index. 
For surface immersions,
Nowik's theory implies that finite-type behavior
with respect to codimension-one transitions
remains of order at most one,
even when the polynomial degree in the Alexander indices is high. 
However, this does not mean that higher polynomial degree carries no information, since the resulting invariants still vary according to the polynomial degree. For this reason, a further development of the latter theory will also be pursued in subsequent work.   

\subsection*{Acknowledgements}
This work was supported by JSPS KAKENHI Grant Numbers JP22K03603 and JP25K06999.  

\bibliographystyle{plain}
\bibliography{QRef}

\begin{bibdiv}
\begin{biblist}

\bib{Goryunov1997}{incollection}{
      author={Goryunov, V.~V.},
       title={Local invariants of mappings of surfaces into three-space},
        date={1997},
   booktitle={The {A}rnold-{G}elfand mathematical seminars},
   publisher={Birkh\"{a}user Boston, Boston, MA},
       pages={223\ndash 255},
         url={https://doi.org/10.1007/978-1-4612-4122-5_11},
      review={\MR{1429894}},
}

\bib{ItoMizuno2026sliceF}{misc}{
      author={Ito, Noboru},
      author={Mizuno, Hiroki},
       title={An explicit slice formula for surface invariants via curve
  invariants},
        date={2026},
        note={Preprint},
}

\bib{Nowik2004}{article}{
      author={Nowik, Tahl},
       title={Order one invariants of immersions of surfaces into 3-space},
        date={2004},
        ISSN={0025-5831,1432-1807},
     journal={Math. Ann.},
      volume={328},
      number={1-2},
       pages={261\ndash 283},
         url={https://doi.org/10.1007/s00208-003-0482-1},
      review={\MR{2030377}},
}

\bib{Shumakovich1995}{article}{
      author={Shumakovich, A.},
       title={Explicit formulas for the strangeness of plane curves},
        date={1995},
        ISSN={0234-0852},
     journal={Algebra i Analiz},
      volume={7},
      number={3},
       pages={165\ndash 199},
      review={\MR{1353494}},
}

\bib{Turaev1994Book}{book}{
      author={Turaev, V.~G.},
       title={Quantum invariants of knots and 3-manifolds},
      series={De Gruyter Studies in Mathematics},
   publisher={Walter de Gruyter \& Co., Berlin},
        date={1994},
      volume={18},
        ISBN={3-11-013704-6},
      review={\MR{1292673}},
}

\end{biblist}
\end{bibdiv}

\end{document}